\documentclass[12pt]{amsart}
\usepackage{amssymb,amsmath,amsfonts,epsfig,latexsym,texdraw}
\usepackage[all]{xy}
 \newtheorem{theorem}{Theorem}[section]
\newtheorem{definition}[theorem]{Definition}
\newtheorem{proposition}[theorem]{Proposition}
\newtheorem{lemma}[theorem]{Lemma}
\newtheorem{corollary}[theorem]{Corollary}

\theoremstyle{definition}

\def\N{\ensuremath{\mathbb{N}}}
\def\Z{\ensuremath{\mathbb{Z}}}
\def\Q{\ensuremath{\mathbb{Q}}}

\def\P{\ensuremath{\mathbb{P}}}

\def\F{\ensuremath{\mathbb{F}}}
\def\C{\ensuremath{\mathbb{C}}}
\def\A{\ensuremath{\mathbb{A}}}
\def\R{\ensuremath{\mathbb{R}}}
\def\L{\ensuremath{\mathbb{L}}}
\def\K{\ensuremath{\mathbb{K}}}

\def\O{\ensuremath{\mathcal{O}}}

\def\cu{\ensuremath{\mathcal{U}}}
\def\ca{\ensuremath{\mathcal{A}}}

\def\ccr{\ensuremath{\mathcal{R}}}
\def\cg{\ensuremath{\mathcal{G}}}
\def\BR{\ensuremath{\mathbf{R}}}
\def\cy{\ensuremath{\mathcal{Y}}}
\def\T{\ensuremath{\mathbb{T}}}

\def\cc{\ensuremath{\mathcal{C}}}
\def\cd{\ensuremath{\mathcal{D}}}
\def\cg{\ensuremath{\mathcal{G}}}

\def\cn{\ensuremath{\mathcal{N}}}
\def\cs{\ensuremath{\mathcal{S}}}

\def\Lhat{\widehat{\L}}

\def\Lbar{\overline{\L}}

\def\<{\ensuremath{\langle}}
\def\>{\ensuremath{\rangle}}

\DeclareMathOperator{\Max}{Max}

\DeclareMathOperator{\Hilb}{Hilb} \DeclareMathOperator{\Hom}{Hom}

\DeclareMathOperator{\Span}{Span}
\DeclareMathOperator{\Spec}{Spec}
\DeclareMathOperator{\Conv}{Conv}

\DeclareMathOperator{\Trop}{Trop}

\DeclareMathOperator{\Chow}{Chow}
\DeclareMathOperator{\an}{{an}}
\DeclareMathOperator{\FC}{FC}
\DeclareMathOperator{\sm}{sm}
\DeclareMathOperator{\spe}{sp}
\DeclareMathOperator{\UH}{UH}
\DeclareMathOperator{\Sch}{Sch}
\DeclareMathOperator{\Rig}{Rig}
\DeclareMathOperator{\Sets}{Sets}
\DeclareMathOperator{\Wt}{Wt}
\def\ol#1{{\overline{#1}}}
\def\init#1#2{{{\operatorname{in}}_{#1}(#2)}}

\def\excise#1{}
\begin{document}

\title{Tropical Realization Spaces for Polyhedral Complexes}

\author[Eric Katz]{Eric Katz}
\address{Department of Combinatorics \& Optimization, University of Waterloo, 200 University Avenue West, Waterloo, ON, Canada N2L 3G1} \email{eekatz@math.uwaterloo.ca}

\begin{abstract}Tropicalization is a procedure that assigns polyhedral complexes to algebraic subvarieties of a torus.  If one fixes a weighted polyhedral complex, one may study the set of all subvarieties of a toric variety that have that complex as their tropicalization.  This gives a  ``tropical realization'' moduli functor.  We use rigid analytic geometry and the combinatorics of Chow complexes as studied by Alex Fink to prove that when the ambient toric variety is quasiprojective, the moduli functor is represented by a rigid space.  As an application, we show that if a polyhedral complex is the tropicalization of a formal family of varieties then it is the tropicalization of an algebraic family of varieties.
\end{abstract}

\maketitle

\section{Introduction}

Let $\K$ be a field complete with respect to a non-trivial non-Archimedean absolute value $|\ |$ and with corresponding valuation, $v=-\log(|\cdot|)$. Let $\Gamma=v(\K^*)$ be the value group.  Given a $d$-dimensional subvariety $Y^\circ$ of the algebraic torus $(\K^*)^n$, tropicalization is a method to associate a weighted $\Gamma$-rational polyhedral complex $\Trop(Y^\circ)$ in $\R^n$.  Throughout the paper, we will use the technical frameworks of Gubler  \cite{GublerGuide} and Osserman-Payne \cite{OP} for the fundamentals of tropicalization and tropical intersection theory, respectively.  Gubler's paper works over general non-Archimedean valued fields while Osserman-Payne's paper requires that $\K$ be algebraically closed.  When we apply the results of Osserman-Payne, we will implicitly pass to the algebraic closure.  This does no harm as the tropicalization is unchanged by valued field extensions \cite[Prop 3.6]{GublerGuide}.
The reader is advised to refer to \cite{GublerGuide} for references to the original sources.   

By the Bieri-Groves theorem, $\Trop(Y^\circ)$ has a possibly non-canonical structure of a purely $d$-dimensional $\Gamma$-rational polyhedral complex \cite[Thm 10.14]{GublerGuide}.   This polyhedral complex is equipped with a {\em multiplicity} (or weight) function $m$ \cite[Sec. 13]{GublerGuide} that assigns a positive integer $m(\sigma)$ to every $d$-dimensional polyhedron $\sigma$.  This multiplicity function satisfies the balancing condition \cite[Sec. 13.9]{GublerGuide}.    

Given a purely $d$-dimensional weighted balanced $\Gamma$-rational polyhedral complex $\cd$ with multiplicity function $m:\cd_{(d)}\rightarrow\N$ on the $d$-dimensional polyhedra in $\cd$, we say that a variety $Y^\circ\subset (\K'^*)^n$ (where $\K'$ is a valued extension of $\K$) is a {\em tropical realization} of $(\cd,m)$ if $\Trop(Y^\circ)$ has underlying set $|\cd|$ and whose multiplicity function is a refinement of $m$.
We will study the set of all realizations of $(\cd,m)$ that are contained in a toric variety $X(\Delta)$.

\begin{definition} An integral closed subscheme $Y$ of $X(\Delta)$ is said to be a {\em tropical realization} of $(\cd,m)$ in $X(\Sigma)$ if $\Trop(Y^\circ)$ has tropicalization $(\cd,m)$ where $Y^\circ=Y\cap (\K^*)^n$.
\end{definition}

\begin{definition}
Let $S$ be a rigid  space.  A {\em family of tropical realizations} of $(\cd,m)$ in $X(\Delta)$ over $S$ is a rigid  subspace $Y$ of $X(\Delta)\times S$, flat and proper over $S$, such that for every $s\in S$, the fiber $Y_s$ is a tropical realization of $(\cd,m)$.
\end{definition}

As we will explain below, because $Y_s$ is proper, by rigid GAGA, it is the analytification of a subscheme of $X(\Delta)$.  We define the tropicalization of $Y_s$ to be the tropicalization of that subscheme.
Because flatness and properness are natural under base-change and tropicalization is unchanged under field extension, the pullback of a family of tropicalization realizations under $S'\rightarrow S$ is also a family of tropical realizations.  Therefore, there is a {\em tropical realization functor}
\[\ccr_{\Delta,\cd,m}:\Rig\rightarrow\Sets\]
taking $S$ to the set of families of tropical realization of $(\cd,m)$ over $S$ in $X(\Delta)$.

\begin{theorem}  \label{t:realspace} If $X(\Delta)$ is a quasiprojective toric variety, then $\ccr_{\Delta,\cd,m}$ is represented by an admissible open subset of the analytification of a scheme of finite type.\end{theorem}

Recall that a toric variety $X(\Delta)$ is quasiprojective if and only if $\Delta$ can be extended to a regular subdivision $\Sigma$ of $N_\R$ \cite[Sec 3.4]{Fulton}. 

Our work is phrased in the language of rigid analytic geometry because fixing the 
tropicalization of a variety $Y\subset (\K^*)^n$ imposes conditions on the central fiber of a model $\cy$ of $Y$ over the valuation ring $\K^\circ$ of $\K$.  Such conditions are not 
algebraic but are rigid analytic instead. For the quasiprojective case, these conditions cut out an open set (called an {\em admissible open}) in a natural $G$-topology on the analytifications of the Hilbert scheme.  This allows us to  apply rigid  techniques to study tropical realizations.

Moreover, one may want to identify the locus of subvarieties of $X(\Delta)$ that satisfy a natural smoothness condition called sch\"{o}nness that was introduced in \cite{TevComp}.  While we do not know  if the points parameterizing sch\"{o}n subvarieties form an admissible open, we are able to show the following weaker result:

 \begin{theorem} \label{t:schon} Every  proper sch\"{o}n subvariety $Y\subset  X(\Delta)$ has an admissible open neighborhood in $\BR_{\Delta,\cd,m}$ consisting of proper sch\"{o}n subvarieties.   \end{theorem}

Our main application of Theorems \ref{t:realspace} and \ref{t:schon} is in the study of tropical realizations.  One may be interested in realizing a complex $(\cd,m)$ by an algebraic family of varieties instead of a variety over a general non-Archimedean field.  We have the following result in this direction:

\begin{theorem} \label{t:approxprinciple} Let $\L$ be a field with non-Archimedean absolute value.  Let $\hat{\L}$ denote its completion with respect to the valuation.  If a weighted rational polyhedral complex $(\cd,m)$ has a tropical realization by an integral subscheme $Y^\circ\subset (\Lhat'^*)^n$ (where $\Lhat'$ is an extension of $\Lhat$), then it has a realization by an integral subscheme ${Y'}^\circ$ defined over the algebraic closure $\ol{\L}$.  Moreover, if $Y^\circ$ is sch\"on then ${Y'}^\circ$ may be chosen sch\"on.
\end{theorem}

This result  follows from the fact that $\BR_{\Delta,\cd,m}$ is an admissible open and from a density argument in rigid geometry.  By setting $\L=\C(t)$ with valuation induced by the ideal $(t)$, we get that a complex $(\cd,m)$ realizable over $\overline{\C((t))}$ is realizable over some finite extension of $\C(t)$.  This shows that formal realizability results also prove algebraic realizability results.  

In the case where $\Delta$ is not quasiprojective, the Hilbert functor of $X(\Delta)$ is represented by an algebraic space, not a scheme.  Its analytification can be studied by techniques of Conrad and Temkin \cite{CTAlgSpace}.  Unfortunately, our arguments which require projectivity do not extend to this case.

The question of realization spaces was studied in the case of varieties over fields with trivial valuations in a paper with Sam Payne \cite{KPReal}.   In that paper, one fixed a simplicial purely $d$-dimensional fan $\Delta$ with a multiplicity function $m$ and defined a tropical realization moduli functor $\BR_{\Delta,m}:\Sch\rightarrow\Sets$.  The moduli functor was shown to be representable by an algebraic space in general and by a scheme of finite type when $X(\Delta)$ is quasiprojective.  The arguments in this paper differ from those in the paper with Payne in that we make very heavy use of the Chow variety to impose conditions on initial degenerations.   In particular, we use a polyhedral complex called the {\em Chow complex} which is a natural but combinatorially more tractable cycle-theoretic analogue of the Gr\"{o}bner complex.   Theorem \ref{t:realspace} rests on a result  proved by Fink in \cite{F} that the Chow complex of a subscheme is determined by its tropicalization. This is used to show that we only need to check finitely many initial degenerations to determine the tropicalization and thus gives us an admissible open condition.

We do not study the existence of tropical realizations here.  The existence of realizations is addressed in upcoming work of Brugall\'{e}-Mikhalkin, papers of Brugall\'{e}-Shaw \cite{BruShaw}, Nishinou \cite{N}, Speyer \cite{Spe07}, Tyomkin \cite{Tyomkin}, with Payne \cite{KPReal} and Bogart \cite{BKLiftingHyper}, and in \cite{KLiftingSpace}.

Most of this paper is taken up with the construction of the realization space.  Sections 2 and 3 give background material in tropical and rigid geometry, respectively.  Sections 4 reviews weight complexes and shows their connection to rigid geometry.  Section 5 provides background about Chow varieties and Chow complexes.  The construction of the realization space as a subspace of the Hilbert scheme is in section 6.  The proof of Theorem \ref{t:schon} is in section 7.  Section 8 establishes Theorems \ref{t:approxprinciple}.

We would like to thank David Helm for explaining the proof of Lemma \ref{l:approx}.
We would also like to thank Alex Fink, Walter Gubler, Sam Payne, Joe Rabinoff, Alan Stapledon, and Ravi Vakil for valuable discussions. 
 
\section{Tropical Geometry}
We review some basic notions of tropical geometry following \cite{GublerGuide}.  For other approaches, one may also consult \cite{FirstSteps, TropicalBook, KTT}.

Let $\K$ be a non-Archimedean valued field with valuation ring $\K^\circ$ and residue field 
$\tilde{\K}$.  In this section, we do not require that $\K$ be complete.  When we pass to an extension of $\K$, we will also use $v$ to denote the extension of 
the valuation.  Let $M$ be a free abelian group of rank $n$ and $N=\Hom(M,\Z)$ be the dual group.  
Let $T=\Spec \K[M]$ be the algebraic torus with character group $M$ and let 
$\T=\Spec \K^\circ[M]$ be the split torus over $\K^\circ$.  Let $M_\R=M\otimes\R$, 
$N_\R=N\otimes\R$.  For $\Delta$, a rational fan in $N_\R$, let $X(\Delta)$ be the associated toric variety compactifying $T$.  Let $v:T\rightarrow N_\R$ be given by
\[t\mapsto (u\mapsto v(\chi^u(t)))\]
where $\chi^u$ is the character associated to $u\in M$.  Once we pick coordinates, we may think of this map as $v:(\K)^*\rightarrow \R^n$, the Cartesian product of valuations.

Given an integral closed $d$-dimensional subscheme $Y^\circ\subset (\K^*)^n$ and $w\in N_\R$, we may define  the {\em initial degeneration} \cite[Sec. 5]{GublerGuide} of $Y^\circ$, $\init{w}{Y^\circ}.$ Let $\K'$ be a valued extension of $\K$ with $w=v(t)$ for some $t\in T_{\K'}$.  Then $\init{w}{Y^\circ}$ is the special fiber of the closure of $t^{-1}Y^\circ_{\K'}$ in the torus $T_{\K'^\circ}$.  It is a closed subscheme of $T_{\tilde{\K}'}$ and is defined only up to translation by an element of $T_{\tilde{\K}'}$ and up to field extension of $\K'$ \cite[Prop 5.5]{GublerGuide}.  

The tropicalization $\Trop(Y^\circ)$ of $Y^\circ$ is a subset of $N_\R$.  It is defined in \cite[Sec. 3]{GublerGuide} as the image of the Berkovich analytification of $Y^\circ$ over the completion of $\K$ under the Cartesian product of valuations.  However, in this paper we will work with the characterization of $\Trop(Y^\circ)$ given by the fundamental theorem of tropical geometry \cite[Thm 5.6]{GublerGuide}: $\Trop(Y^\circ)$ is the equal to the set $\{w\in N_\R|\init{w}{Y^\circ}\neq\emptyset\}$.  The set $\Trop(Y^\circ)$ has the structure of a purely $d$-dimensional $\Gamma$-rational polyhedral complex.  

We refer to \cite[Appendix]{GublerGuide} for facts about polyhedral complexes although our notation is different.  A purely $d$-dimensional polyhedral complex $\cd$ is said to be {\em weighted} if there is a multiplicty function $m:\cd_{(d)}\rightarrow\N$ assigning a positive integer to every top-dimensional polyhedron.  For a polyhedron $\sigma$, we denote the relative interior of $\sigma$ by $\sigma^\circ$.  For a polyhedral complex $\cd$, we use $|\cd|$ to denote the union of the underlying polyhedra of $\cd$.  We say a complex $\cd$ is supported on a complex $\cc$ if $|\cd|\subset|\cc|$.
A complex $\cd'$ is said to be a refinement of a complex $\cd$ if $|\cd'|=|\cd|$ and every polyhedron of $\cd'$ is contained in a polyhedron of $\cd$.

There is a polyhedral complex structure on $\Trop(Y^\circ)$ that is induced from the Gr\"{o}bner complex \cite[Sec. 10]{GublerGuide}.  It has the following property: for $\sigma$, a cell of $\Trop(Y^\circ)$, let $T_\sigma\subset T$ be the subtorus whose cocharacter space is the span of $\sigma-w$ for $w\in\sigma^\circ$; for $w\in\sigma^\circ$, $\init{w}{Y^\circ}$ is invariant under multiplication by $T_\sigma$.  
This polyhedral complex is equipped with a {\em multiplicity} (or weight) function $m:\Trop(Y)_{(d)}\rightarrow\N$ \cite[Sec. 13]{GublerGuide}.  This multiplicity $m(\sigma)$ for a $d$-dimensional polyhedron $\sigma$ is defined as the sum of the multiplicities of $\init{w}{Y^\circ}$ over its irreducible components over an algebraically closed field for some $w\in\sigma^\circ$. This multiplicity function is independent of the choice of $w\in\sigma^\circ$ and 
satisfies the balancing condition \cite[Sec. 13.9]{GublerGuide}.   Consequently, the multiplicity function is locally constant on the regular points of $\Trop(Y)$, that is, the points $w$ for which there is a polytope $\sigma\subset\Trop(Y)$ such that $\sigma^\circ$ is a neighborhood of $w$ in $\Trop(Y^\circ)$.  Because the polyhedral structure on $\Trop(Y^\circ)$ is not unique, we will work with polyhedral complexes on $\Trop(Y^\circ)$ up to refinement.  The multiplicity function on a refinement is the one induced from that of the coarser subdivision.

We will make use of integral subschemes $Y$ of a toric variety $X(\Delta)$.  For such a subscheme, we will write $Y^\circ$ for $Y\cap T$.   We define initial degenerations: for $w\in N_\R$, pick a valued extension $\K'$ of $\K$ and $t\in T_{\K'}$ with $v(t)=w$; then 
$\init{\Delta,w}{Y}$ is the special fiber of the closure of $t^{-1}Y_{\K'}$ in $X(\Delta)_{\K'^\circ}$.
It will not, in general, 
be true that $\init{\Delta,w}{Y}$ is equal to $\overline{\init{w}{Y^\circ}}$.  This is because the first 
subscheme may have components supported on the toric boundary of $X(\Delta)$.  It is true that 
$\init{w}{Y^\circ}=\init{\Delta,w}{\overline{Y}}\cap T_{\tilde{\K}}$ \cite[Sec 10.11]{GublerGuide}.

There is a natural notion of tropical smoothness for varieties $Y^\circ\subset T$ introduced by Tevelev \cite{TevComp} called {\em sch\"{o}nness}.  We will use the following characterization:
$Y^\circ$ is sch\"{o}n if and only if $\init{w}{Y^\circ}$ is smooth for all $w\in\Trop(Y^\circ)$.  See \cite[Prop 3.8]{HK} for a proof.

We also define initial degenerations of points in projective space.  Suppose $\T$ acts linearly on 
some projective space $\P^{N'}_{\K^\circ}$.  For $w\in N_\R$ pick a valued field extension $\K'$ of $\K$ and $t\in T_{\K'}$ satisying $v(t)=w$, then for any $\mathbf{x}\in\P^{N'}(\K')$ define $\init{w}{\mathbf{x}}$ to be the
intersection  of $\overline{t^{-1}\cdot\mathbf{x}}\subset\P^{N'}_{\K'^\circ}$  with the special fiber, 
$\P^{N'}_{\tilde{\K}'}$.  This initial degeneration is well-defined only up to $T_{\tilde{\K}'}$-action.

\section{Rigid Analytic Geometry}

\subsection{Basic Notions}
We introduce some basic notions from rigid analytic geometry.   As references we recommend \cite{BGR,Conrad,FvdP,Schneider}.  

Let $A$ be an affinoid algebra and $\Max A$ be its maximal spectrum.  
A rational subset $U\subset \Max A$ is one defined by
\[U=\{x|\ |f_i(x)|\leq |g(x)| \text{ for } i=1,\dots,s\}\]
where $g,f_1,\dots,f_s\in A$ generate the unit ideal.  A set of the form 
\[\{x\in \Max A|\ |f_i(x)|\leq 1,\ |g_i(x)|\geq 1\}\]
for $f_1,\dots,f_r,g_1,\dots,g_s\in A$ is called a {\em Laurent domain}.

$\Max A$  carries a natural G-topology (which is a Grothendieck topology whose open sets are particular subsets of $M(A)$ under inclusion).  The open sets and covers are called {\em admissible sets} and {\em admissible covers}.  We will not use admissible sets in any deep way and only need the following observations:
\begin{enumerate}
\item Zariski open sets are admissible,

\item Finite unions of rational opens and Laurent domains are admissible,

\item  Finite intersection of admissible sets are admissible,

\item A subset defined by $\{x\ \big|\ |f(x)|<b\}$ is admissible.
\end{enumerate}

A rigid space over $\K$ is a locally ringed G-topologized space $(X,\O_X)$ that is in a particular sense, locally modeled on affinoid spaces with their structure sheaves. 

There is an {\em analytification} functor from the category of separated schemes of finite type over $\K$ to rigid spaces \cite[Ex 4.3]{FvdP}.  \excise{We first explain the analytification of projective space.  Let $[Z_0:\dots:Z_n]$ be homogeneous coordinates on $\P^n$.  We consider the standard open $U_i=\Max \K\<\frac{Z_0}{Z_i},\frac{Z_1}{Z_i},\dots,\frac{Z_n}{Z_i}\>$.  Note that the underlying set of $U_i$ consists of the polydisc $B^n$.  In other words, we have 
\[\left|\frac{Z_0}{Z_i}\right|,\dots,\left|\frac{Z_n}{Z_i}\right|\leq 1.\]
so the standard open $U_i$ corresponds to points in $\P^n$ with the maximum of $|
Z_k|$ achieved by $|Z_i|$.}  If $X\subset\P^n$ is a projective variety, then one can take the
analytification $X^{\an}$ by analytifying neighborhoods $V_i$  in standard affine charts to $V_i^{\an}$ and gluing.
Let $\pi\in\K$ with $|\pi|<1$. For $V_i=\Spec\left(\K[x_1,\dots,x_n]/(f_1,\dots,f_k)\right)$, one defines
$V_{i.l}^{\an}=\Max \K\<\pi^lx_1,\dots,\pi^lx_n\>/(f_1,\dots,f_k)$
to be the intersection of the analytification with the polydisc of radius $|\pi|^{-l}$.  If $V_{i,l}$ is any affinoid chart of $X^{\an}$, then any rational subset of $V_{i,l}$ is an admissible set.

 If $X\subset\P^n$ is a projective flat scheme over $\O$ (in particular if $X$ is base-changed from a projective scheme over $\Z$), there is a specialization map $\spe:X^{\an}\rightarrow X_{\tilde{\K}}$  \cite[Sec. 4.8]{FvdP}.  For a $\K$-point of $X$, $\spe(x)$ is the intersection of $\overline{x}$ with the closed fiber $X_{\tilde{\K}}$.  
 
If $X=\P^{N'}$ is a projective space with a linear $T$-action induced from an action of $\T$ on $\P^{N'}_{\K^\circ}$, then specialization is closely related to initial degenerations: for $w\in N_\R$, pick a valued field extension $\K'$ of $\K$ and $t\in T_{\K'}$ with $v(t)=w$; then for $\mathbf{x}\in\P^n(K')$, $\init{w}{\mathbf{x}}=\spe(t^{-1}\cdot \mathbf{x})$.  
 
The following lemma is standard (combine, say, \cite[Prop 0.2.3]{Berthelot},\cite[Sec. 7.7]{FvdP}) and is proved by looking at affinoid charts and considering lifts of the defining equations of a subset.

\begin{lemma} \label{l:inverse special} Let $X\subset\P^n$ be a projective flat scheme over $\K^\circ$.  If $Y$ is a Zariski locally closed subset of $X_{\tilde{\K}}$, then $\spe^{-1}(Y)$ is an admissible open in $X^{\an}$. 
\end{lemma}

\excise{This can be shown by considering the covering of $X$ by standard opens in $\P^{N'}$.  If $f_1,\dots,f_k\in\K\<x_1,\dots,x_n\>$ cut out a closed subscheme $V$ in $X_k\cap U_i$, then $\spe^{-1}(V)\cap U_i$ is given by $|f_i|,\dots,|f_n|<1$ which is admissible.  Likewise, if $O=X_k\setminus V$, $\spe^{-1}(O)$ is cut out by $|f_1|=\dots=|f_n|=1$ which defines a rational open in the analytification of an affine open in $X$, hence an admissible open in $X^{\an}$.  }

Because locally closed subsets play an important role here, we will find the following simple topological lemma useful:

\begin{lemma} \label{l:locallyclosed} Let $f:X\rightarrow Y$ be a proper continuous map of topological spaces.  Suppose $V\subset X$ is locally closed and satisfies $f^{-1}(f(V))=V$.  Then $f(V)$ is locally closed.
\end{lemma}

\begin{proof}
Write $V=Z_1\setminus Z_2$ for $Z_1,Z_2$ closed.  One observes $f(Z_1\setminus Z_2)=f(Z_1)\setminus f(Z_2)$ and so is locally closed.
\end{proof}

\subsection{Tropicalization of rigid subspaces}

We will define the tropicalization of proper rigid subspaces $Y$ of toric varieties $X(\Delta)$.  While the tropicalization of rigid subspaces was introduced by Gubler \cite{Gubler}, a theory of multiplicities on the tropicalization has not been developed.  For our situation, we will restrict ourselves to proper subspaces of toric varieties.  In this case, we can apply rigid GAGA, and then tropicalize the relevant subscheme.  For the proof of rigid GAGA, one establishes the correspondence between rigid and algebraic coherent sheaves \cite[Sec. 7.4]{Abbes}, and then applies the proof of the existence theorem for subschemes given for formal schemes  \cite[$\text{III}_1$, 5.1-5.2]{EGA}  to the rigid setting.  

We will make use of the following:
\begin{lemma} (compare \cite[$\text{III}_1$,\ Cor (5.1.8)]{EGA}) Let $X$ be a separated scheme of finite type over an affine scheme $S$.  Analytification induces a bijection between the set of closed subschemes, proper over $S$ of $X$ and the set of closed rigid subspaces of $X^{\an}$, proper over $S$.
\end{lemma}

\begin{definition} The {\em tropicalization} of a proper rigid subspace $Y$ of $X(\Delta)^{\an}$ is $\Trop(Y\cap T)$ where $Y$ is considered as a subscheme of $X(\Delta)$.
\end{definition}

\subsection{Hilbert functor}

In \cite[Sec. 4]{CRelAmple}, Conrad constructs the Hilbert functor for rigid spaces.  Let $f:X\rightarrow S$ be a proper morphism of rigid spaces,.  The Hilbert functor $\Hilb(X/S)$ classifies closed immersions $Z\hookrightarrow X\times_S T$ such that $Z$ is flat over $T$ for rigid spaces $T$ over $S$.  If $L$ is an $S$-ample invertible sheaf on $X$, then the Hilbert functor is the disjoint union of subfunctors $\Hilb_Q(X/S)$ for Hilbert polynomials $Q\in\Q[t]$.  In this case, by \cite[Thm 4.1.3]{CRelAmple}, the Hilbert functor is represented by a proper rigid space.  Moreover, the construction of the Hilbert scheme commutes with analytification: the Hilbert functor for Hilbert polynomial $Q$ is represented by the analytification of the Hilbert scheme $\Hilb_Q(X/S)$.

\section{Weight Subdivisions}
We review some notions of weight complexes and subdivisions.  These are straightforward generalizations of dual fans and weight polytopes as described in \cite{GKZ}.   Our exposition follows \cite[Sec. 9]{GublerGuide}.

We will study the case where $\T=\Spec \K^\circ[M]$ acts linearly on some projective space, $\P^{N'}_{\K^\circ}$.  Here, $\P^{N'}$ is not necessary the completion of $T=\Spec \K[M]$.
By \cite[Lem 9.7]{GublerGuide}, such an action lifts to a linear representation of $\T$ on $\A^{N'+1}_{\K^\circ}$.  By the arguments of \cite[Prop 9.8]{GublerGuide}, one may pick a $\K^\circ$-basis $v_0,\dots,v_{N'}$ of $(\K^\circ)^{N'+1}$ that gives a  simultaneous eigenbasis for the $T$-action on $\K^{N'+1}$.  Let $\chi_0,\dots,\chi_{N'}$ (possibly with repetitions)  be the characters of $v_0,\dots,v_{N'}$.  Given $\mathbf{x}\in \P^{N'}(\K)$, we may pick a lift $x\in \K^{N'+1}$ and write
\[x=\sum_i c_iv_i.\]
Let $\ca$ be the set of characters $\chi_i$ for which $c_i\neq 0$.  The {\em weight polytope} associated to $\mathbf{x}$ is $\Wt(\mathbf{x})=\Conv(\ca)$.  Define the height function $a:\ca\rightarrow\Gamma$ by
\[a(\chi)=\min \{v(c_i)|\chi_i=\chi\}.\]
It is independent of the choice of lift $x$ up to addition by a global constant and is independent of the choice of simultaneous eigenbasis.  The upper hull associated to $x$, $\UH_x$ is the convex hull of $\{(\chi,\lambda)\in M_\R\times \R|\chi\in\ca, \lambda\geq a(\chi)\}$.  The images of its faces under the projection $\pi:M_\R\times\R\rightarrow M_\R$ gives the {\em weight subdivision} of $\Wt(\mathbf{x})$.

Note that $\R$-translates of the upper hull induce the same subdivsion.  Points $\chi\in\ca$ for which $(\chi,a(\chi))$ lies on the face of $\UH_x$ are said to be vertices of the face in the weight subdivision.

\begin{definition} \label{d:wtcomplex}
The {\em weight complex} $\cc$ of $x$ is the linearity complex of the piecewise-linear function
\[F(w)=\min_{\chi\in \ca}(\<\chi,w\>+a(\chi)).\]
That is, it is the complete polyhedral complex in $N_\R$ whose $n$-dimensional cells are the domains of linearity of $F(w)$.   
\end{definition}

The weight complex and subdivision are independent of the choice of lift $x$.  In fact, once $Wt(\mathbf{x})$ is known, the weight complex $\cc$ determines the upper hull up to $\R$-translates \cite[Thm 2.5]{F}.

When the weight complex is given the polyhedral complex structure induced by the function $F(w)$, there are bijective correspondences \cite[Prop 9.12]{GublerGuide} between 
\begin{enumerate}
\item polyhedra of the weight complex, 
\item faces of the weight subdivision, and 
\item $T_{\tilde{\K}}$-orbits of the special fiber of $\overline{T \cdot {\mathbf x}}\subset\P^{N'}_{\K^\circ}$
\end{enumerate}

To a face $Q$ of the weight subdivision, we associate the polyhedron $\sigma$ of the weight complex for which the minimum of $\<\chi,w\>+a(\chi)$ is achieved exactly at the vertices $\chi$ of $Q$ for all $w\in\sigma^\circ$.  This corresponds to the orbit containing initial degenerations of the form $\init{w}{\mathbf{x}}$ for $w\in\sigma^\circ$.  The characters for which $\init{w}{\mathbf{x}}$ has a non-zero component in the eigenbasis are exactly the vertices of $Q$.  To a polyhedron $\sigma$ of the weight complex corresponds an algebraic subtorus $T_\sigma\subset T_{\tilde{\K}}$ characterized by the cocharacter space $(N_\sigma)_\R\subset N_\R$ of $T_\sigma$ being equal to $\Span(\sigma-w)$ for $w\in\sigma^\circ$.  The initial degeneration $\init{w}{\mathbf{x}}$ is invariant under $T_\sigma$ for $w\in\sigma^\circ$ \cite[Lem 4.21]{KTT}.

Let $\partial\Wt({\mathbf{x}})$ be the extremal vertices of $\Wt({\mathbf{x}})$.  Let $\P^*$ be the Zariski open subset of $\P^{N'}$ such that for all $\chi\in\partial\Wt({\mathbf{x}})$, $c_i\neq 0$ for some $i$ with $\chi_i=\chi$.  

\begin{lemma} \label{l:wtcomplex} The set $U_\cc$ of all $\mathbf{x}\in (\P^*)^{\an}$ with fixed weight complex $\cc$ is an admissible open.
\end{lemma}

\begin{proof}

Let $UH\subset M_\R\times \R$ be the upper hull inducing the weight complex $\cc$ such that $\pi_\R(\UH)=[0,\infty)$.  Pick an ordering $\{\chi_i\}$ of elements of $\ca$.
Let $v_{\chi,1},\dots,v_{\chi,r_\chi}$ for varying $\chi$ be a $\T$-eigenbasis for $(\K^\circ)^{N+1}$.  Write an element $x\in (\K)^{N+1}$ as $x=\sum_{\chi_i,j} c_{\chi_i,j}v_{\chi_i,j}$.

 Let $(\chi_{\min},0)$ be a point of the upper hull with minimum $\R$-coordinate.   Therefore, any $\mathbf{x}$ with weight complex $\cc$ can be lifted to $x\in\K^{N+1}$ satisfying $|c_{{\chi_i},j}|\leq 1$ for all $i,j$ and $c_{{\chi_{\min},j}}=1$ for some $j$. Therefore, $U_\cc$ is a finite union of its intersections with $\{c_{{\chi_{\min},j}}=1\}$ for various values of $j$.
Thus, we may work in a standard affinoid chart of $\P^{N'}$ with $c_{\chi_{\min},j}=1$.

Now, if $(\chi,a({\chi}))$ is a vertex of the upper hull, then we have
\[\max(|c_{{\chi},1}|,\dots,|c_{\chi,d_\chi}|)=e^{-a(\chi)}.\]
This condition can be rewritten as
\[\left(\bigcap_l \{|c_{\chi,l}|\leq e^{-a(\chi)}\}\right)\cap \left(\bigcup_l \{|c_{\chi,l}|= e^{-a(\chi)}\}\right).\]
We repeat this for all vertices of the upper hull.  If $\chi\in Q$ is not a vertex of the weight subdivision, then it is in the relative interior of a cell with vertices $\chi_1,\dots,\chi_k$ giving an integer convex dependency, $m\chi=\sum m_i\chi_i$.  The condition that $(\chi,v(c_{\chi,l}))$ lies above or on that that cell of the upper hull translates to $|c_{\chi,l}|\leq e^{-(\sum m_i h_i)/m}$ for all $l$ which is an admissible condition.  The intersection of all these conditions is a finite union of Laurent domains, hence admissible.
\end{proof}

\section{Chow varieties and tropicalization}

In this section, we give an approach to tropicalization based on the Chow variety rather than the Hilbert scheme.   We recommend \cite[Ch. I. 3]{Kollar} or \cite[Ch. 4]{GKZ} for foundational facts about the Chow variety.   We follow the convention of \cite{GKZ} in that we do not normalize the Chow variety.  The approach here closely follows \cite[Sec. 5]{KTT}.  

\subsection{Chow varieties}
Let $X(\Sigma)$ be a projective toric variety with a $T$-equivariant projective embedding $i:X(\Sigma)\hookrightarrow\P^{n}$.  The universal bundle $\O(1)$ on $\P^{n}$ pulls back to a very ample line bundle on $X(\Sigma)$.  For a $d$-dimensional subvariety $V$ of $X(\Sigma)$, we define the degree $\deg(V)$ to be the integer $\deg(c_1(\O(1))^d\cap [V])$.  We can extend this degree linearly to purely $d$-dimensional algebraic cycles.  

The Chow variety $\Chow_{d,d'}(X(\Sigma))$ parameterizes purely $d$-dimensional effective cycles in $X(\Sigma)$ of degree $d'$.  In fact, it coarsely represents a particular moduli functor of effective cycles \cite[I.4.14]{Kollar}.
Moreover, it is a closed subscheme of $\Chow_{d,d'}(\P^{n})$ which, in turn,  is a subscheme of some $\P^{N'}$.  The construction of the point $R_V$ of $\Chow_{d,d'}(\P^n)$ corresponding to a $d$-dimensional closed integral subscheme $V$ of degree $d'$ is as follows: one takes the locus of hyperplanes $(H_0,\dots,H_d)$ in $((\P^n)^\vee)^{d+1}$ such that the intersection $V\cap H_0\cap\dots\cap H_d$ is non-empty; this is a hypersurface in $((\P^n)^\vee)^{d+1}$ of multi-degree $(d',\dots,d')$; its defining equation is the {\em Chow form} $R_V$.  Now, the $T$-action on $\P^n$ induces an action on $((\P^n)^\vee)^{d+1}$ and hence on the defining equations of hypersurfaces in $((\P^n)^\vee)^{d+1}$.  Let $\P^{N'}=\P(\Gamma(\O(d',\dots,d')))$, the space of multi-degree $(d',\dots,d')$ hypersurfaces in $((\P^n)^\vee)^{d+1}$.
It follows that the $T$-action on $X(\Sigma)$ induces an action on $\Chow_{d,d'}(X(\Sigma))$ which extends to $\P^{N'}$.  The definition of Chow forms extends to cycles as follows:  the Chow form of a cycle $Z=\sum m_i[V_i]$ is $R_Z=\prod R_{V_i}^{m_i}$.
The addition of cycles induces a morphism  (see \cite[Lem 3.24.2]{Kollar})
\[u:\Chow_{d,d'_1}(X(\Sigma))\times \Chow_{d,d'_2}(X(\Sigma))\rightarrow \Chow_{d,d'_1+d'_2}(X(\Sigma)).\]

Since the Chow form respects flat degenerations \cite[(I.3.23.1.7)]{Kollar}, taking the Chow form commutes with initial degenerations in the following sense:
\[\init{w}{R_V}=R_{\init{\Sigma,w}{V}}.\]

Let $\Chow(X(\Sigma))$ be the disjoint union of $\Chow_{d,d'}(X(\Sigma))$ taken over all $d,d'$.
There is a natural $T$-equivariant {\em fundamental cycle} morphism $\FC:\Hilb(X(\Sigma))\rightarrow \Chow(X(\Sigma))$ taking a scheme to the underlying cycles of its top-dimensional components \cite[I.6.3]{Kollar}.

\subsection{The Chow complex}

Let $X(\Sigma)\subset\P^n$ be a $T$-equivariantly embedded projective toric variety.  Let $d,d'$ be non-negative integers.   We may perform the construction of the Chow variety over $\K^\circ$.  This induces an action of $\T$ of $\P^{N'}_{\K^\circ}$ and allows us to apply the results in the previous section.  

\begin{definition}
Let $Y\subset X(\Sigma)$ be a $d$-dimensional closed subscheme of degree $d'$.
The weight complex corresponding to the Chow form $R_Y\in\P^{N'}$
is called the {\em Chow complex.}  It is given the polyhedral structure induced by the piecewise-linear function $F(w)$ as in Definition \ref{d:wtcomplex}.  
\end{definition}

The Chow complex was introduced in \cite[Sec. 5]{KTT} as a valued field analogue of the secondary fan \cite{GKZ}, but the first in-depth study of it was by Fink \cite{F}.  It was shown that the Gr\"obner complex is a refinement of the Chow complex in \cite[Prop 5.12]{KTT}.   Examples of Chow complexes of hypersurfaces and linear subspaces are discussed in \cite{F}.  For the special case of a projectively embedded toric variety with the embedding defined over $\Z$, the Chow complex is the well-known secondary fan \cite{GKZ}.   For our purposes, we do not need to determine the Chow complex of any cycle.   Instead, we will make use of the following characterization of the Chow complex: 
two points $w,w'$ are in the relative interior of the same polyhedron of the Chow complex if and only if $\init{w}{R_Y}$ and $\init{w'}{R_Y}$ are $T_{\tilde{\K}}$-translates.
Because $R_{\init{\Sigma,w}{Y}}=\init{w}{R_Y}$, the underlying cycles of $\init{\Sigma,w}{Y}$ and $\init{\Sigma,w'}{Y}$ are $T_{\tilde{\K}}$-translates.
  
\begin{lemma} The complex $\Trop(Y^\circ)$ is supported on the $d$-skeleton of the Chow complex.
\end{lemma}

\begin{proof}
We adapt the proof of \cite[Thm 10.14]{GublerGuide} which is the analogous statement for the Gr\"{o}bner complex.   For ease of notation, before taking an initial degeneration $\init{\Sigma,w}{Y}$, we we will suppose that we have extended $\K$ to a field $\K'$ such that $T_{\K'}$ has elements of valuation $w$.

We must show that any polyhedron of the Chow complex whose relative interior intersects $|\Trop(Y^\circ)|$ is, in fact, contained in $|\Trop(Y^\circ)|$. Let $\sigma$ be a polyhedron of the Chow complex and $w\in\sigma^\circ\cap |\Trop(Y^\circ)|$.  
By the fundamental theorem of tropical geometry, $w\in\Trop(Y^\circ)$ if and only if 
$\init{\Sigma,w}{Y}\cap T_{\tilde{\K}}\neq \emptyset$.   By \cite[Prop 11.3]{GublerGuide},
$\init{\Sigma,w}{Y}$ is purely $d$-dimensional.  Consequently, the components of 
$\init{\Sigma,w}{Y}$ intersecting $T_{\tilde{\K}}$ have a non-trivial underlying cycle. 
From our characterization of the Chow complex, if $w'\in\sigma^\circ$ then the underlying cycle of  $\init{\Sigma,w'}{Y}$ is a $T_{\tilde{\K}}$-translate of the underlying cycle of $\init{\Sigma,w}{Y}$.  Therefore, its support intersects $T_{\tilde{\K}}$.  It follows that 
$\init{\Sigma,w'}{Y}$ intersects $T_{\tilde{\K}}$, and so 
$w'\in|\Trop(Y^\circ)|$.
\end{proof}

It follows that the multiplicity function is constant on $d$-dimensional polyhedra of the Chow complex that intersect $\Trop(Y^\circ)$.  
We will make use of the following result of Fink that lets us recover the Chow complex from the tropical variety. This result is a systematization of the method of orthant shoting introduced as \cite[Thm 2.2]{DFS}.  To prove orthant shooting in our context, one applies the proof in \cite[Sec. 10]{KTT}, justifying the intersection theory in our more general situation by \cite[Sec. 5]{OP}.

\begin{proposition}\cite[Thm 5.1]{F}\label{p:fink} Let $Y$ be a $d$-dimensional integral subscheme in $X(\Sigma)_{\K}$ intersecting $T$.  The Chow complex $\cc$ of $Y$ is determined by $|\Trop(Y^\circ)|$ and the multiplicity function $m$.
\end{proposition}

Fink gives a concise formula for the codimension $1$ skeleton of the Chow complex in terms of an operation he calls stable Minkowski sum.   From this skeleton and knowledge of the Chow polytope, one recovers the Chow complex.  We do not need to make use of Fink's explicit formula here.

Since $\Trop(Y^\circ)$ determines the Chow complex and $\Trop(Y^\circ)$ is supported on the $d$-skeleton of the Chow complex, once we have chosen a fan $\Sigma$, we may suppose that the weighted polyhedral structure $(\cd,m)$ on $|\Trop(Y^\circ)|$ is a refinement of the $d$-skeleton of the Chow complex.

\subsection{Tropicalization via the Chow variety}

Because the tropicalization of a variety $V$ only depends on its underlying cycle, we can give an alternative approach to tropicalization using the Chow form.  This formulation is directly analogous to that of \cite{GublerGuide} except that instead of the Hilbert point of $V$, we use the Chow form of $V$ and instead of the Gr\"{o}bner complex, we use the Chow complex.

Let $X(\Sigma)\subset \P^{n}$ be a $T$-equivariantly embedded projective toric variety.  Let $\partial X(\Sigma)$ denote the toric boundary of $X(\Sigma)$.  Let $i:\partial X(\Sigma)\rightarrow X(\Sigma)$ be the natural inclusion.  Under this inclusion, we have a projective embedding, $\partial X(\Sigma)\subset\P^n$.
Now, let 
\[i_*:\Chow_{d,d'}(\partial X(\Sigma))\rightarrow \Chow_{d,d'}(X(\Sigma))\]
be the induced map of Chow varieties \cite[Thm 6.8]{Kollar} that takes a cycle supported on $\partial X(\Sigma)$ to that same cycle to be considered as a cycle on $X(\Sigma)$.  
For a degree $d'$, let 
\[\Chow_{d,d'}(X(\Sigma))^{\partial}=\bigcup_{1\leq e'\leq d'} u(i_*(\Chow_{d,e'}(\partial X(\Sigma)))\times\Chow_{d,d'-e'}(X(\Sigma))).\]
This subscheme parameterizes purely $d$-dimensional cycles of degree $d'$ that have some component supported on the toric boundary.
Because $u$ is proper, this subscheme is Zariski closed.
Define $\Chow_{d,d'}(X(\Sigma))^\circ$ by 
\[\Chow_{d,d'}(X(\Sigma))^\circ=\Chow_{d,d'}(X(\Sigma))\setminus\Chow_{d,d'}(X(\Sigma))^\partial.\]
This is the Zariski open subset parameterizing cycles in $\Chow_{d,d'}(X(\Sigma))$ all of whose components intersect $T$.
Now, let $\Chow_{d,d',e'}(X(\Sigma))$ be the subscheme given by
\[\Chow_{d,d',e'}(X(\Sigma))=u(\Chow_{d,e'}(X(\Sigma))^\circ\times \Chow_{d,d'-e'}(\partial X(\Sigma))).\]
It parameterizes cycles of degree $d'$ that can be expressed as the sum of a degree $e'$ cycle all of whose components intersect $T$ and a degree $d'-e'$ cycle supported on $\partial X(\Sigma)$.  By applying Lemma \ref{l:locallyclosed} to the map $u:\Chow_{d,e'}(X(\Sigma))\times \Chow_{d,d'-e'}(X(\Sigma))\rightarrow \Chow_{d,d'}(X(\Sigma))$, we see that it is locally closed.  

Given an integral subscheme $Y$ of $X(\Sigma)$, the Chow form $R_Y$ contains information about the underlying cycle of $Y$.  This information is enough to determine $\Trop(Y^\circ)$.  The underlying set of $\Trop(Y^\circ)$ is the set of all $w$ such that the initial degeneration $\init{\Sigma,w}{Y}$ intersects $T_{\tilde{\K}}$.   The multiplicity function on $\Trop(Y^\circ)$ is characterized as follows \cite[Sec. 13]{GublerGuide}: for $w\in\sigma^\circ$ for $\sigma$, a top-dimensional polyhedron of $\Trop(Y^\circ)$, the cycle of the initial degeneration can be decomposed as 
\[[\init{w}{Y^\circ}]=\sum m_i [V_i] \] 
over an algebraically closed field where the $V_i$'s are the prime components; then $m(w,Y^\circ)=\sum m_i$; this turns out to depend only on $\sigma$ by the choice of polyhedral structure on $\Trop(Y^\circ)$.
Now, for a subtorus $T'\subset T$, we define the degree of $T'$ to be
\[\deg(T')=\deg(\overline{T'\cdot x})\]
where $x\in T\subset X(\Sigma)$. Because $\O(1)$ is very ample on $X(\Sigma)$, we may ensure that the intersection that computes the degree of $\overline{T'\cdot x}$ takes place in $T'\cdot x$.  Consequently, the degree is independent of the choice of $x$.

\begin{proposition} \label{p:chowcheck} Let $Y$ be an $d$-dimensional closed integral  subscheme of $X(\Sigma)$ of  degree $d'$.  Let $(\cd,m)$ be a weighted purely $d$-dimensional polyhedral subcomplex of a refinement of the $d$-skeleton of the Chow complex of $Y.$  Then 
$\Trop(Y^\circ)=(\cd,m)$ if and only if for all $w\in N_\R$, 
\begin{enumerate}
\item \label{i:supset} $w\in |\cd|$ if and only $\init{w}{R_Y}\not\in \Chow_{d,d'}(\partial X(\Sigma))$;
\item \label{i:mult} if $w\in \sigma^\circ$ for $\sigma\in \cd_{(d)}$ then $\init{w}{R_Y}\in \Chow_{d,d',e'}(X(\Sigma))$ where $e'=m(\sigma)\deg(T_\sigma)$.
\end{enumerate}
\end{proposition}

\begin{proof}
Now, $\init{\Sigma,w}{Y}$ is purely $d$-dimensional.  
By the fundamental theorem of tropical geometry, $w\in\Trop(Y^\circ)$ if and only if $\init{\Sigma,w}{Y}\cap T_{\tilde{\K}}\neq \emptyset$.  Therefore, $w\in\Trop(Y^\circ)$ is equivalent to  the components of $\init{\Sigma,w}{Y}$ intersecting $T_{\tilde{\K}}$ having a non-trivial underlying cycle.  Since $\init{w}{R_Y}=R_{\init{\Sigma,w}{Y}}$, this condition is equivalent to $\init{w}{R_Y}\not\in \Chow_{d,d'}(\partial X(\Sigma))$.  Therefore, $w\in \Trop(Y^\circ)$ if and only if $\init{w}{R_Y}\not\in \Chow_{d,d'}(\partial X(\Sigma))$.  We can conclude that (\ref{i:supset}) is equivalent to $|\Trop(Y^\circ)|=|\cd|$.

Let $w\in\sigma^\circ$ for $\sigma\in\cd_{(d)}$.  Let 
 \[Z=\sum m_i [V_i]\]
 be the sum of the cycles in $\init{\Sigma,w}{Y}$ that intersect $T_{\tilde{\K}}$.  
Because $w$ is in the relative interior of a $d$-dimensional polyhedron of the Chow complex,  $\init{w}{R_Y}$ is invariant under the $d$-dimensional torus $T_\sigma$.   It follows that the cycle $\init{\Sigma,w}{Y}$ and therefore each $V_i$ is invariant under $T_\sigma$.  Since each $V_i$ is $d$-dimensional, each $V_i$ is the closures of a translate of $T_\sigma$.
 Because $\init{w}{Y^\circ}\cap T_{\tilde{\K}}=\init{w}{Y^\circ}$, $Z$ is the cycle of $\overline{\init{w}{Y^\circ}}$.  Therefore, the tropical multiplicity at $w$ is given by $m(w,Y^\circ)=\sum m_i$.   It follows that  
 \[\deg(\overline{\init{\Sigma,w}{Y}\cap T})=\sum m_i\deg(V_i)=\sum m_i \deg(T_\sigma)=m(w,Y^\circ)\deg(T_\sigma).\]
Consequently (\ref{i:mult}) is equivalent to $m(w,Y^\circ)=m(\sigma)$ for all $w\in\sigma^\circ$.
\end{proof}

\begin{corollary} \label{c:pointchowcheck} Let $Y$ be a $d$-dimensional subvariety of $X(\Sigma)$.   Let $(\cd,m)$ be a purely $d$-dimensional weighted subcomplex of a refinement of the $d$-skeleton of the Chow complex of $Y$.
For each polyhedron $\sigma$ of $\cc$,  let $w_\sigma$ be a point in the relative interior of $\sigma$.  If the two conditions in Proposition \ref{p:chowcheck} are satisfied for all $w_\sigma$, then $\Trop(Y^\circ)=(\cd,m)$.
\end{corollary}

\begin{proof} We note that if $w\in\sigma^\circ$ for $\sigma\in\cc$ then $\init{w}{R_Y}$ and $\init{w_\sigma}{R_Y}$ are $T_{\tilde{\K}}$-translates.  The two conditions of the proposition are $T_{\tilde{\K}}$-invariant.  Since the support of $\cc$ is $N_\R$, it follows that the conditions hold for every $w\in N_\R$.
\end{proof}

\section{Construction of the Realization Space}

In this section we construct the realization space of the balanced weighted rational polyhedral complex $(\cd,m)$.  We compactify $X(\Delta)$ to a projective toric variety $X(\Sigma)$.  This allows us to look at the Hilbert scheme parameterizing subschemes of $X(\Sigma)$.  The set of Hilbert points of schemes with fixed Chow complex will be shown to be an admissible open.  Then we may pick a point $w_\sigma$ in the relative interior of each polyhedron in $\cc$.  We impose the conditions in Proposition \ref{p:chowcheck} at each $w_\sigma$.  This is again an admissible open.  The intersection of all such admissible opens is our desired realization space by Corollary \ref{c:pointchowcheck}.

Pick a projective completion $X(\Sigma)$ of $X(\Delta)$ and a $T$-equivariant embedding $i:X(\Sigma)\hookrightarrow\P^n$.  Note that $X(\Delta)$ is a $T$-equivariant open subvariety of $X(\Sigma)$.

\begin{lemma}Let $(\cd,m)$ be a  weighted  polyhedral complex.  Then there is a finite type subscheme of $\Hilb(X(\Sigma))$ containing the Hilbert point $[Y]$ for any tropical realization $Y$ of $(\cd,m)$ in $X(\Sigma)$.
\end{lemma}

\begin{proof}
Let $V$ be a projective subspace of $\P^n$ of complementary dimension to $Y$.   The tropicalization of $i(Y)^\circ$, considered as a subscheme of the dense torus of $\P^n$ is determined by $\Trop(Y^\circ)$ by \cite[Thm 13.17]{GublerGuide}.  By \cite[Thm 5.3.3]{OP} and the Kleiman-Bertini \cite{KleimanBertini} theorem, one may pick $V$ such that the intersection is
 transverse, is supported in $T$, and satisfies 
 $\Trop(V^\circ\cap i(Y)^\circ)=\Trop(V^\circ)\cdot\Trop(i(Y)^\circ)$.  Therefore, by \cite[Thm 5.1.1]{OP}, the intersection number of 
 $i(Y)$ and $V$ is determined by their tropicalizations.  Consequently, the degree $d'$ of $Y$ is 
 determined by $(\cd,m)$.  By the proof of \cite[Thm 3.2]{KPReal}, 
 there are finitely many possible 
 Hilbert polynomials for $Y$.  The desired subscheme of $\Hilb(X(\Sigma))$ is the union of the components corresponding to these Hilbert polynomials.
\end{proof}

Within $\Hilb_Q(X(\Sigma))$, the set of Hilbert points of integral closed subschemes contained in $X(\Delta)$ and intersecting $T$ is an open subscheme which we call $U_Q$.  Let $\cc$ be the Chow complex of a variety with tropicalization $(\cd,m)$. 

\begin{lemma} \label{l:hilbpoints} Let $U_\cc$ be the set of all points in $U_Q^{\an}$ with Chow complex $\cc$.  Then $U_\cc$ is an admissible open.
\end{lemma}

\begin{proof}
The Hilbert polynomial $Q$ determines the dimension $d$ and degree $d'$ of the subschemes it parameterizes.  One has the fundamental cycle morphism
$\FC:\Hilb_Q(X(\Sigma))\rightarrow\Chow_{d,d'}(X(\Sigma))$.  
 By taking the composition of the fundamental cycle morphism $\FC^{\an}:\Hilb(X(\Sigma))^{\an}\rightarrow \Chow(X(\Sigma))^{\an}$ with the projective embedding 
 \[\Chow_{d,d'}(X(\Sigma))^{\an}\hookrightarrow(\P^{N'})^{\an},\]
  we obtain a map $\Hilb(X(\Sigma))^{\an}\rightarrow (\P^{N'})^{\an}$.  The points of $(\P^{N'})^{\an}$ with weight complex $\cc$ is an admissible open by Lemma \ref{l:wtcomplex}.  We can pull back that admissible set to $\Hilb_Q(X(\Sigma))^{\an}$ and intersect with $(U^Q)^{\an}$
\end{proof}

\begin{lemma} \label{l:chowcheckadmis} Let $w\in N_\R$.  The set $U_{Q,w}$ of all  $[Y]$ of $U_Q^{\an}$ satisfying the conditions of Proposition \ref{p:chowcheck} for $w$ is an admissible open.
\end{lemma}

\begin{proof}
Let $\K'$ be a valued field extending $\K$ such that there exists $t\in T(\K')$ with valuation $w$.  Then the initial degeneration (defined up to $T(\tilde{\K}')$-action), $\init{\Sigma,w}{Y}$ has Chow form $\spe(t^{-1}R_Y)\in\Chow_{d,d'}(X(\Sigma))_{\tilde{\K}'}$.  The conditions in Proposition \ref{p:chowcheck} correspond to $\spe(t^{-1}R_Y)$ lying in a particular locally closed subset.  The inverse image of a locally closed subset under $\spe^{-1}$ is admissible.  The inverse image of this admissible subset under the map $\FC^{\an}$ is the desired set.
\end{proof}

We now complete the proof of Theorem \ref{t:realspace}:

\begin{proof}
The Chow complex $\cc$ of an integral, closed subscheme of $X(\Sigma)$ is determined by its
 tropicalization $(\cd,m)$.  We may suppose that $(\cd,m)$ is a refinement of the $d$-skeleton of 
 $\cc$.  For every polyhedron $\sigma$ of the $d$-skeleton of $\cc$, pick a point 
 $w_\sigma\in\sigma^\circ$.  Now, for every possible Hilbert polynomial $Q$ that can occur for a
realization of $(\cd,m)$, let $U_{Q,\cc}$ be the open subscheme of $\Hilb_Q(X(\Sigma))$ produced by
Lemma \ref{l:hilbpoints}.  Within $U_Q^{\an}$, consider the admissible open 
$R_Q=U_{Q,\cc}\cap\bigcap_\sigma U_{Q,w_\sigma}$.  Let $\BR_{\Delta,\cd,m}$ be the union of the sets $R_Q$ over Hilbert polynomials $Q$.   This set consists of the Hilbert points of all 
realizations of $(\cd,m)$ in $X(\Sigma)$ by Corollary \ref{c:pointchowcheck}.  Because $(\Hilb(X(\Sigma))^{\an}$ represents the functor of rigid flat families in $X(\Sigma)$ by \cite[Thm 4.1.3]{CRelAmple}, $\BR_{\Delta,\cd,m}$ represents the functor $\ccr_{\Delta,\cd,m}$.
 \end{proof}

\section{Sch\"{o}n locus}

In this section, we prove Theorem \ref{t:schon} which says that every sch\"{o}n realization of $(\cd,m)$ 
has an admissible open neighborhood of sch\"{o}n subvarieties in  $\BR_{\Delta,\cd,m}$.  As above, pick an equivariantly embedded projective toric variety $X(\Sigma)\subset\P^n$ containing $X(\Delta)$ as an open $T$-invariant subvariety.

Let $\Hilb_Q^{\sm}(X(\Sigma))$ be the subfunctor of $\Hilb_Q(X(\Sigma))$ parameterizing flat families of subschemes $Y\rightarrow S$ of $X(\Sigma)$ such that $Y\cap T$ is smooth over $S$.  

\begin{lemma} \label{l:smoothininterior} There is a finite union of locally  closed subschemes in $\Hilb_Q(X(\Sigma))$, $L$ such that for any field $\F$, $\Hilb_Q^{\sm}(X(\Sigma))(\F)=L(\F)$ 
\end{lemma}

\begin{proof}
Consider the universal family $p:\cu\rightarrow\Hilb_Q(X(\Sigma))$.  Let $\cn\cs\subset\cu$ be the non-smooth locus of $p$.  $\cn\cs$ is closed \cite[Cor 2.12]{Liu}. Let 
\[\cn\cs^\partial=\cn\cs\cap \left(\partial X(\Sigma)\times\Hilb_Q(X(\Sigma))\right)\]
be the scheme-theoretic intersection of the non-smooth locus with the toric boundary of $X(\Sigma)$.  $\F$-points of $\Hilb_Q^{\sm}(X(\Sigma))$ are the points of $\Hilb_Q(X(\Sigma))$ over which $\cn\cs=\cn\cs^\partial$.

We apply a flattening stratification argument \cite[Ch. 8]{MumfordCurves} to find a decomposition of $\Hilb_Q(X(\Sigma))$ into a finite set $S_1,\dots,S_m$ of locally closed subschemes such that the fiber products
 \[\cn\cs\times_{\Hilb_Q(X(\Sigma) )} S_k,\ \cn\cs^\partial\times_{\Hilb_Q(X(\Sigma) )} S_k\]
 are flat over $S_k$ for all $k$.  There is a closed, possibly empty subscheme $L_k$ of $S_k$ where $\cn\cs$ and $\cn\cs^\partial$ agree.  In fact, consider the maps $i,i^\partial:S_k\rightarrow \Hilb(X(\Sigma))$ given by the flat families $\cn\cs\rightarrow S_k,\cn\cs^\partial\rightarrow S_k$; the locus $L_k$ is where $i,i^\partial$ agree.
 It follows that over $L_k$, $\cn\cs$ is disjoint from $T\times L_k$.
Let $L=L_0\cup L_1\cup\dots\cup L_m$.
Any $\F$-point of $\Hilb_Q^{\sm}(X(\Sigma))$ clearly belongs to some $L_k$.
\end{proof}
Note that $L$ does not necessarily represent the functor $\Hilb_Q^{\sm}(X(\Sigma))$ as taking the flattening stratification may change the moduli problem.

We will need to make use of Gr\"{o}bner complex \cite[Sec. 10]{GublerGuide}, the weight complex associated to the Hilbert point $[Y]\in\Hilb_Q(X(\Sigma))$ of a subscheme $Y\subset X(\Sigma)$.  Two points $w,w'$ are in the same cell of the Gr\"{o}bner complex if and only if the Hilbert points $[\init{\Sigma,w}{Y}]$ and $[\init{\Sigma,w'}{Y}]$ are $T_{\tilde{\K}}$-translates.  This happens if and only if the subschemes $\init{\Sigma,w}{Y}$ and $\init{\Sigma,w'}{Y}$ are $T_{\tilde{\K}}$-translates.  

\begin{lemma} \label{l:smoothdeg} Let $w\in N_\R$.  The subset  $S_w$ of $\Hilb(X(\Sigma))^{\an}$ such that $\init{w}{Y^\circ}$ is smooth is an admissible open.
\end{lemma} 

\begin{proof} 
After a possible valued field extension of $\K$, we have an initial degeneration map  $\operatorname{in}_w:\Hilb(X(\Delta))^{\an}\rightarrow\Hilb(X(\Delta))_{\tilde{\K}}$ given by $[Y]\mapsto \spe(t^{-1}[Y])$ for some $t\in T$ with $v(t)=w$.  The set $S_w$ is the inverse image of the Zariski closed subset $L$ in Lemma \ref{l:smoothininterior}.  This is an admissible open.
\end{proof}

Now, for each polyhedron $\sigma$ of the Gr\"{o}bner complex of $Y$ let $w_\sigma$ be a point in the relative interior of $\sigma$. 

\begin{lemma} \label{l:smoothcheck}   
 If $\init{w_\sigma}{Y^\circ}$ is smooth for all $\sigma$ then $Y^\circ$ is sch\"{o}n.
\end{lemma}

\begin{proof} 
If $w$ is in the relative interior of a polyhedron $\sigma$ of the Gr\"{o}bner complex, then after possibly passing to a valued field extension $\K'$ of $\K$,  $\init{w}{Y^\circ}$ and $\init{w_{\sigma}}{Y^\circ}$  are $T_{\tilde{\K}'}$-translates.
Consequently, $\init{w}{Y^\circ}$ is always smooth and $Y^\circ$ is sch\"{o}n.
\end{proof}

We now prove Theorem \ref{t:schon}: 

\begin{proof}
Let $Y\subset X(\Sigma)$ be such a sch\"{o}n tropical realization of $(\cd,m)$ .  We construct an 
admissible neighborhood of $[Y]$ in $\BR_{\Delta,\cd,m}$ such that every subvariety in that 
neighborhood is sch\"{o}n.  Let $\cg$ be the Gr\"{o}bner complex of $V$.  The set of all elements of 
$\BR_{\Delta,\cd,m}$ with Gr\"{o}bner complex $\cg$ is an admissible open $U_\cg$ by a proof 
directly analogous to Lemma \ref{l:hilbpoints}.  Construct the sets $S_{w_\sigma}$ by 
Lemma \ref{l:smoothdeg}.  The desired set is $U_\cg\cap \bigcap_\sigma S_{w_\sigma}$.  All 
subschemes corresponding to points of that set are sch\"{o}n by Lemma \ref{l:smoothcheck}.
 \end{proof}

We do not know if the set of all sch\"{o}n subvarieties is an admissible open.  The issue is that there may be infinitely many possibilities for the Gr\"{o}bner complex of a variety with fixed tropicalization.  This is not true for the Chow complex because it is determined by the tropicalization.  

\section{Formal and algebraic realizations}

In this section, we show that the existence of tropical realizations of $(\cd,m)$ over an extension of the completion of a field imply the existence of realizations over the algebraic closure of the original field.  Specifically, we will let $\L$ be a field with a non-trivial non-Archimedean valuation.  Let $\Lhat$ be the completion of $\L$ with respect to the valuation and $\Lbar$ be the algebraic closure of $\L$.   Let $\Lhat'$ be some extension of $\Lhat$. Our main result is that the existence of a realization over $\Lhat'$ implies the existence of a realization over $\Lbar$.   The case we have in mind for applications is when $\L=\C(t)$ with valuation induced by uniformizer $t$ and $\Lhat=\C((t))$.  A realization over $\Lbar$ is defined over some finite extension of $\C(t)$ which is the function field of some curve.  This implies that $(\cd,m)$ is realizable as an algebraic family over some curve.  

We make use of the following density lemma whose proof was explained to us by David Helm:

\begin{lemma} \label{l:approx} Let $X$ be an algebraic variety over $\L$ and $X^{\an}=(X\times_\L \Lhat)^{\an}$.  Let $U$ be a non-empty admissible open in $X^{\an}$.  Then $U$ contains a $\Lbar$-point of $X$.
\end{lemma}

\begin{proof}
By shrinking $X$ and $U$, we may suppose that $X$ is affine and $U$ is rational.  Since we are only concerned with finding $\Lbar$-points on $X$, we may replace $X$ by the reduction of one of its components.  Let $\overline{X}$ denote the closure of $X$ in $\P^n\supset\A^n$.    Shrink $U$ to ensure that it is disjoint from $\overline{X}\setminus X$.
 By considering coordinate projections, we can produce a finite surjective morphism $p:\overline{X}\rightarrow \P^d$ as in, for example \cite[2.11]{dJ}.    Pick a standard affine $\A^d\subset \P^d$ such that its analytification intersects $p(U)$, and further shrink $X$ such that $p$ maps $X$ to $\A^d$.  By shrinking $U$, we may suppose that $p$ maps $U$ into the unit polydisc 
 \[B_1^d=\{x_1,\dots,x_n|\ |x_1|,\dots,|x_n|\leq 1\}\]
  of $(\A^d)^{\an}$.  Let $V=p^{-1}(B_1^d)$.  Then $p:V\rightarrow B_1^d$ is a finite surjective map of affinoids.  By \cite[Prop 8.1.2(2)]{FvdP}, $p(U)$ is a finite union of rational opens of $B_1^d$.  Each rational open must contain a $\Lbar$-point.  A preimage of such a point in $U$ is the desired point.
 \end{proof}

This yields the proof of Theorem \ref{t:approxprinciple}:

\begin{proof}
If $(\cd,m)$ is realizable by $Y\subset X(\Delta)_{\Lhat'}$, then $\BR_{\Delta,\cd,m}$, as constructed 
above, is non-empty. It is an admissible open in a finite union of components of 
$\Hilb(X(\Delta))^{\an}$. It must contain a $\Lbar$-point, $[Y']$ by the above lemma. Then $Y'$, considered as a subscheme of $X(\Delta)$ is defined over $\Lbar$.  The subscheme $Y'^\circ=Y\cap T$ has tropicalization $(\cd,m)$.

If $Y^\circ$ is sch\"on, apply the above argument to the admissible neighborhood constructed by Theorem \ref{t:schon}.
\end{proof}

\bibliographystyle{plain}

\end{document}